\documentclass[11pt,a4paper]{article}
\usepackage{amsmath,amsfonts,enumerate,amssymb,graphicx}
\usepackage[dvips]{epsfig}
\usepackage{pgf,tikz}
\usepackage{graphicx,float,anysize}
\usepackage{array,fullpage,color,cite}
\usepackage{amsmath,amssymb,amsfonts,amsthm,float,color,enumerate}
\newtheorem{girth}{Theorem}
\newtheorem{proposition}{Proposition}

\newcommand*{\diam}{\mathrm{diam}}
\newcommand*{\radius}{\mathrm{rad}}

\author{Ratko Darda\thanks{University of Primorska, UP FAMNIT, Glagolja\v ska 8, 6000 Koper, Slovenia. E-mail: {\tt ratko.darda@student.upr.si}}}

\begin{document}

\title{There is no square-complementary graph of girth $6$} \maketitle

\begin{abstract}
A graph is {\it square-complementary} ({\it squco}, for short) if its square and complement are isomorphic.
We prove that there is no squco graph of girth $6$, thus
answersing a question asked by Milani\v c et al.~[Discrete Math., 2014, to appear],
and leaving $g = 5$ as the only possible value of $g$ for which the existence of
a squco graph of girth $g$ is unknown.
\end{abstract}

\section{Introduction}

Given two graphs $G$ and $H$, we say that $G$ is the {\it square} of $H$ (and denote this by $G = H^2$) if their vertex sets coincide and two distinct vertices $x$, $y$ are adjacent in $G$ if and only if $x$, $y$ are at distance at most two in $H$. Squares of graphs and their properties are well-studied in literature (see, e.g., Section 10.6 in the monograph~\cite{MR1686154}).
A  graph $G$ is said to be {\it square-complementary} ({\it squco} for short) if its square is isomorphic to its complement. That is, $G^2\cong \overline{G}$, or, equivalently, $G\cong \overline{G^2}$.
The question of characterizing squco graphs was posed by Seymour Schuster at a conference in 1980~\cite{Schuster}. Since then, squco graphs were studied in the context of graph equations in terms of operators such as the line graph and complement (see \cite{MR613401, MR1322267, MR1405844, MR725876, MR1018613, MR1379114}).
The entire set of solutions of some of these equations was found (see for example \cite{MR613401} and references quoted therein).
However, the set of solutions of the equation $G^2 \cong \overline{G}$ remains unknown, despite several attempts to describe it (see for example \cite{MR1322267, MR1405844,MilPedPelVer}). The problem of determining all squco graphs was also posed as Open Problem No.~36 in Prisner's book~\cite{MR1379114}.

Examples of squco graphs are $K_1$, $C_7$, and a cubic vertex-transitive bipartite squco graph on $12$ vertices, known as the Franklin graph (see Fig.~\ref{fig:FranklinsGraph}).
\begin{figure}[!ht]
  \begin{center}
\includegraphics[height=25mm]{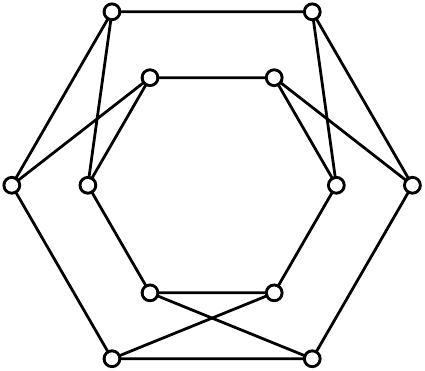}
  \end{center}
  \caption{The Franklin graph.}
\label{fig:FranklinsGraph}
\end{figure}

The following two propositions, due to Balti\'c et al.~\cite{MR1322267} (and partially due to Capobianco and Kim~\cite{MR1405844}),
summarize the results regarding the connectivity, radius, and diameter of squco graphs.

\begin{proposition}\label{prop:cut-vertex}
Every squco graph is connected and has no cut vertices.
\end{proposition}

\begin{proposition}\label{prop:radius-diameter}
If $G$ is a nontrivial squco graph, then  $\radius(G) = 3$ and $3\le \diam(G)\le 4\,.$
Moreover, if $G$ is regular, then $\diam(G) = 3$.
\end{proposition}

It is not known whether a squco graph of diameter $4$ exists. In the paper~\cite{MilPedPelVer},
several other questions regarding squco graphs were posed, and a summary of the known necessary conditions for squco graphs was given.
Among them is the following result expressing a condition on the girth. (Recall that the {\it girth} of a graph $G$ is the length of a shortest cycle in $G$, or $\infty$ if $G$ is acyclic.)

\begin{proposition}\label{prop:girth-7}
If $G$ is a nontrivial squco graph with girth at least $7$, then $G$ is the $7$-cycle.
\end{proposition}

This proposition leaves only $5$ possible values for the girth $g$ of a squco graph $G$, namely $g\in \{3,4,5,6,7\}$.
The case $g = 7$ is completely characterized by Proposition~\ref{prop:girth-7}.
Balti\'c et al.~\cite{MR1322267} and Capobianco and Kim~\cite{MR1405844} asked whether there exists a squco graph of girth $3$.
An affirmative answer to this question was provided in~\cite{MilPedPelVer} by a squco graph on $41$ vertices with a triangle (namely, the circulant $C_{41}(\{4, 5, 8, 10\})$). As shown by the Franklin graph, there also exists a squco graph of girth $4$.
The questions regarding the existence of squco graphs of girth $5$ or $6$ were left as open questions in~\cite{MilPedPelVer}.
In this note, we answer one of them, by proving that there is no squco graph of girth $6$.
This leaves $g = 5$ as the only possible value of $g$ for which the existence of
a squco graph of girth $g$ is unknown.

We briefly recall some useful definitions.
Given two vertices $u$ and $v$ in a connected graph $G$, we denote by
$d_G(u,v)$ the distance in $G$ between $u$ and $v$ (that is, the number of edges on a shortest $u$-$v$ path).
For a positive integer $i$, we denote by $N_i(v,G)$ the set of all vertices
$u$ in $G$ such that $d_G(u,v)= i$, and by $N_{\ge i}(v,G)$ the set of
all vertices $u$ in $G$ such that $d_G(u,v)\ge i$.  
We use standard graph terminology~\cite{MR2744811}.

\section{The result}

\begin{girth}
There is no squco graph of girth $6$.
\end{girth}

\begin{proof}
Suppose for a contradiction that $G$ is a squco graph of girth $6$.
First, we observe that if $x$ is a vertex of $G$, then there are no edges in any of sets $N_i(x,G)$ for $i=1,2$ and no two distinct vertices in $N_1(x,G)$ have a common neighbor in $N_2(x,G)$.
Let $k = \Delta(G)$ be the maximum degree of $G$, and let $w$ be a vertex of degree $k$.
Since the only squco graphs with maximum degree at most $2$ are $K_1$ and $C_7$~\cite{MilPedPelVer}, we have $k\ge 3$.

We consider two cases.

{\it Case 1.}  $w$ has a neighbor of degree at least three.

Let $v$ be a neighbor of $w$ of degree at least three, and let $p$ and $q$ be two neighbors of $v$ other than $w$. If one of them, say $p$, is of degree at least $3$, then $p$ has at least two
 neighbors in $N_2(v,G)$ and thus  $\Delta(\overline{G^2})\ge |N_1(q,\overline{G^2})|\geq k+1$, contrary to the fact that $\overline{G^2}\cong G$.
 Hence, both $p$ and $q$ are of degree~$2$. (Notice that Proposition~\ref{prop:cut-vertex} excludes the possibility of having degree~$1$ vertices).
 Let $a$ and $b$ be the unique neighbors of $p$ and $q$ in $N_2(v,G)$, respectively.
The set $N_3(v,G)$ is nonempty, because radius of $G$ is $3$ by Proposition~\ref{prop:radius-diameter}.
Vertices $a$ and $b$ must be adjacent to all of vertices in $N_3(v,G)$, otherwise
 $\Delta(\overline{G^2})\ge
 \max\{|N_1(p,\overline{G^2})|,|N_1(q,\overline{G^2})|\}\geq k+1$, contrary to the fact that $\overline{G^2}\cong G$.
To avoid a $4$-cycle in $G$, we conclude that  $|N_3(v,G)| = 1$.
But now, the degree of $v$ in $\overline{G^2}$ is $1$, which implies that $\overline{G^2}$  has a cut vertex, contrary to the fact that
$G$ is squco and Proposition~\ref{prop:cut-vertex}.

{\it Case 2.} All neighbors of $w$ are of degree at most two.

In this case, all neighbors of $w$ are of degree exactly two. In particular, $|N_2(w,G)| = |N_1(w,G)|= k\ge 3$.
Now we will show that every vertex $x$ from $N_2(w,G)$ is of degree at least $|N_3(w,G)|$.
Let $x\in N_2(w,G)$, and let $y$ be the unique neighbor of $x$ in $N_1(w,G)$.
Vertex $x$ has at least $|N_3(w,G)|-1$ neighbors in $N_3(w,G)$, since otherwise $|N_1(y,\overline{G^2})|\geq k+1$.
This implies that any two vertices from $N_2(w,G)$ (the size of $N_2(w,G)$ is at least $3$) have at least $|N_3(w,G)|-2$
common neighbors in $N_3(w,G)$. This bounds $|N_3(w,G)|\leq 3$, otherwise we would have a $4$-cycle.

Suppose $|N_3(w,G)|=3$. To each of the three pairs of vertices in $N_3(w,G)$, associate, if possible,
their common neighbor in $N_2(w,G)$. Because, each vertex in $N_2(w,G)$ is connected to at least two vertices in $N_3(w,G)$, it is surely associated with some pair.  If $|N_2(w,G)|\geq 4$ then some two vertices from $N_2(w,G)$ are associated with the same pair and we get a $4$-cycle, a contradiction.
We thus have $|N_1(w,G)|=|N_2(w,G)|= k\leq 3$ and $|N_{\geq 4}(w,G)|=0$
(otherwise we would have a vertex of degree at least $4>k$ in $\overline{G^2}$).
This implies that our graph has at most ten vertices. All squco graphs with at most $11$ vertices
are known~\cite{MilPedPelVer}; none of them has girth $6$. Hence this is a contradiction with $G$ having girth $6$.

Suppose $|N_3(w,G)|= 2$. If $k\leq 4$, then we our graph has no more than $11$ vertices, which is not possible.
Hence $k\geq 5$. There must be at least $2k-1$ vertices of degree two in $G$ (all $k$ vertices in $N_1(w,G)$; at most one of $k$ vertices in $N_2(w,G)$ has both vertices from $N_3(w,G)$ for neighbors, otherwise we have a $4$-cycle as before). In $\overline{G^2}$ at most $k+3$ of them are of degree two, because every vertex in $N_1(w,G)$ will be connected to all but one vertex in $N_2(w,G)$ in $\overline{G^2}$,which is a contradiction, because $k\geq 5$.

The last possibility is that $|N_3(w,G)|=1$, but then $w$ would be of degree $1$ in $\overline{G^2}$, again a contradiction. This completes the proof.
\end{proof}

\bibliography{squco-bib}{}
\bibliographystyle{plain}

\end{document}